\documentclass[12pt]{amsart} \usepackage{amssymb} \usepackage{mathabx} \usepackage{bm}

\newcommand{\wrlab}[1]{\label{#1}}

\newtheorem{thm}{THEOREM}[section]

\newtheorem{lem}[thm]{LEMMA} 
\newtheorem{cor}[thm]{COROLLARY} 
 
 \newtheorem{thm*}{THEOREM}[]
 \newtheorem{ex}[thm]{EXAMPLE}

\newcommand{\tref}[1]{Theorem~\ref{#1}}
\newcommand{\cref}[1]{Corollary~\ref{#1}}

\newcommand{\lref}[1]{Lemma~\ref{#1}}

\def\N{{\mathbb N}}  
\def\R{{\mathbb R}} \def\C{{\mathbb C}}

 \def\scrf{{\mathcal F}} 
\def\scri{{\mathcal I}}

\def\scru{{\mathcal U}} \def\scrv{{\mathcal V}} \def\scrw{{\mathcal W}}
  
  \def\scrc{{\mathcal C}}

\def\bfa{{\mathbf A}} \def\bfb{{\mathbf B}} \def\bfc{{\mathbf C}}
 \def\bfg{{\mathbf G}}

\def\bfs{{\mathbf S}}  \def\bfu{{\mathbf U}}
\def\bfv{{\mathbf V}} \def\bfw{{\mathbf W}} \def\bfx{{\mathbf X}}
\def\bfy{{\mathbf Y}} \def\bfz{{\mathbf Z}}

  \def\bflitf{{\mathbf f}}

\def\bfalpha{{\bm\alpha}}

\font\tenolde=eufm10 at 10pt
\font\sevenolde=eufm7
\font\fiveolde=eufm5
\newfam\oldefam
\textfont\oldefam=\tenolde
\scriptfont\oldefam=\sevenolde
\scriptscriptfont\oldefam=\fiveolde

  \def\Lg{{\mathfrak g}}

\def\spec{\hbox{spec\,}}
\def\Sm{\hbox{Sm\,}}
\def\Irr{\hbox{Irr\,}}
\def\dom{\hbox{dom\,}}
\def\im{\hbox{im\,}}
\def\dim{\hbox{dim\,}}

\def\Stab{\hbox{Stab}}
\def\GL{\hbox{GL}}
\def\Ad{\hbox{Ad}}

\def\ds{\displaystyle}

\begin{document}
\bibliographystyle{alpha}


\title[Discrete free-Abelian central stabilizers]
{Discrete free-Abelian central stabilizers in a higher order frame bundle}


\author{Scot Adams}
\address{School of Mathematics\\ University of Minnesota\\Minneapolis, MN 55455
\\ adams@math.umn.edu}

\date{June 1, 2017\qquad Printout date: \today}

\begin{abstract}
  Let a real Lie group $G$ have a $C^\infty$ action on a real manifold~$M$.
  Assume every nontrivial element of $G$   has nowhere dense fixpoint set in $M$.
  First, we show, in every frame bundle, except possibly the $0$th,
  that each stabilizer admits no nontrivial compact subgroups.
  Second, we show that, if $G$ is connected,
  then there is a~dense open $G$-invariant subset of some
  higher order frame bundle of~$M$ such that,
  for any point $x$ in that subset, the stabilizer in $G$ of~$x$
  is a discrete, finitely-generated, free-Abelian, central subgroup of~$G$.
  We derive several corollaries of these two results.
\end{abstract}

\maketitle

 

\section{Introduction\wrlab{sect-intro}}

P.~Olver's freeness conjecture (in his words) asserts: ``If a Lie
group acts effectively on a manifold, then, for some $n<\infty$, the
action is free on [a nonempty] open subset of the jet bundle of order $n$.''
There is some ambiguity in this: No mention is made of connectedness
of the group or manifold, the particular choice of jet bundle is not made precise,
and the smoothness of the action is left unspecified.
Here we focus on~higher order frame bundles and $C^\infty$ actions.
The manifold on~which the group acts is not assumed connected.
Some of our results require the Lie group to be connected, but others do not.
To accommodate the $C^\infty$ hypothesis (instead of $C^\omega$),
we~assume that each nontrivial group element has nowhere dense fixpoint set,
a stronger condition than effectiveness.
Under these assumptions, from earlier work, we have:
\begin{itemize}
\item[(1)] In a dense open subset of~a higher order frame bundle,
           stabilizers are discrete (Theorem 6.4 of \cite{olver:movfrmsing}, Theorem 3.3 of \cite{adams:genfreeacts}).
\item[(2)] The conjecture is not true if the Lie group is connected and has~noncompact center
           (\cite{adams:cinftyctx} and \cite{adams:comegactx}).
\item[(3)] The conjecture is true generically (\cite{adams:genfreeacts} and \cite{adolv:generic}).
\item[(4)] The conjecture is true for any smoothly real algebraic action;
           in fact, for any smoothly real algebraic action,
           there exists a~nonempty open invariant subset of a higher order frame bundle
           on which the prolonged action is free and proper (\cite{adams:algacts}).
\end{itemize}

By (2), the freeness conjecture is not true in complete generality.
However, our main result, \tref{thm-freeAb-cent-stabs}, shows that, if the Lie group is connected,
then we do get significant control over the stabilizers in~prolongations;
specifically,
on a dense open subset of some frame bundle,
every stabilizer is a discrete, finitely-generated, free-Abelian central subgroup
of~the ambient Lie group.
\tref{thm-freeAb-cent-stabs} yields:
\begin{itemize}
\item[(A)] The frame bundle conjecture is true if the ambient Lie group
is~connected and has compact center (\cref{cor-cpt-center}).
\end{itemize}
By (2) and (A), we know for exactly which connected Lie groups
the frame bundle freeness conjecture holds: Those with compact center.

By elementary methods,
in \lref{lem-no-cpt-subgps-in-stabs},
we observe that,
in any~frame bundle except~the $0$th,
every stabilizer has no nontrivial compact subgroups.
As a consequence, we prove the following:
\begin{itemize}
\item[(B)] The frame bundle conjecture is true if the original action, before prolongation,
has compact stabilizers (\cref{cor-cpt-stabs-prolong}).
\item[(C)] Any proper action induces a free proper action on~every frame bundle
except possibly the $0$th (\cref{cor-prolong-proper}).
\item[(D)] If a prolonged action is proper, then it is free (\cref{cor-proper-prolong-strong}).
\end{itemize}

A key result of \cite{adams:algacts} is that, for any real algebraic action with compact stabilizers,
there exists a nonempty open invariant set on which the action is proper.
In~\lref{lem-alg-cpt-stabs-to-proper} below, we improve this result from ``nonempty open'' to~``dense open''.
Then, by \lref{lem-no-cpt-subgps-in-stabs},
we are similarly able to~improve (4) from ``nonempty open'' to ``dense open'':
\begin{itemize}
\item[(E)] For any smoothly real algebraic action, there exists a dense open invariant subset
of a higher order frame bundle on which the prolonged action is free and proper (\tref{thm-mainalg}).
\end{itemize}

Any free proper action has a transversal, a.k.a.~a moving frame.
So, by~(4) and (E), we find moving frames
in sufficiently high prolongations of~any smoothly real algebraic action.
Similarly, by (C), we find moving frames
in all prolongations (except possibly the $0$th) of any proper action.
Some nonproper actions have a~proper prolongation to a frame bundle,
and, by (D), any such prolongation admits a moving frame.

Section 6 is independent of the rest of this paper.
It motivates some of the ideas in Section 7, but is not,
strictly speaking, needed.
Section~7 is also independent of the others.
The main results are in Section~8.
Other useful (but easier) dynamical ideas appear in~Sections 4 and 5.

\section{Global notation and conventions\wrlab{sect-global}}

Let $\N:=\{1,2,3,\ldots\}$ and let $\N_0:=\N\cup\{0\}$.
For all $c\in\N$, let $0_c:=(0,\ldots,0)\in\R^c$.
For any set $S$, let $\#S\in\N_0\cup\{\infty\}$ denote the~number of elements in $S$.
For any function $\psi$,
the domain of $\psi$ will be denoted $\dom[\psi]$
and the image of $\psi$ will be denoted $\im[\psi]$.
For any function $\psi$, for any set $S$,
we define
\begin{eqnarray*}
\psi_*(S)&:=&\{\,\,\psi(s)\,\,|\,\,s\in S\cap(\dom[\psi])\,\,\}\qquad\hbox{and}\\
\psi^*(S)&:=&\{\,\,x\in \dom[\psi]\,\,|\,\,\,\psi(x)\in S\,\,\}.
\end{eqnarray*}
For any sets $X$ and $Y$,
a {\bf partial function from $X$ to $Y$}, denoted $\psi:X\dashrightarrow Y$,
is a function $\psi$ s.t.~ both $\dom[\psi]\subseteq X$
and $\im[\psi]\subseteq Y$.

For any group~$G$, the identity element of $G$ is denoted $1_G$.
For any group~$G$, the center of $G$ is denoted $Z(G)$.
By ``Lie group'', we mean real Lie group, unless otherwise stated.
By ``manifold'', we mean $C^\infty$~real manifold, unless otherwise stated.
By ``action'', we mean left action, unless otherwise stated.
An action of a Lie group $G$ on~a manifold $M$ will be called $C^\omega$
if there is a real analytic manifold structure on~$M$,
compatible with its $C^\infty$~manifold structure,
with respect to~which the $G$-action on $M$ is real analytic.
Any subset of~a topological space acquires, without comment,
its relative topology, inherited from the ambient topological space.
For any group $G$, for any set $X$, for any action of $G$ on $X$, for any~$x\in X$,
we define $\Stab_G(x):=\{g\in G\,|\,gx=x\}$.

Let a group $G$ act on a topological space $X$.
The $G$-action on $X$ is~{\bf fixpoint rare} if,
for all $g\in G\backslash\{1_G\}$,
the set $\{x\in X\,|\,gx=x\}$ has empty interior in $X$,
{\it i.e.}, if no nontrivial element of $G$ fixes
every point in a nonempty open subset of~$X$.
If $G$ is a connected locally compact Hausdorff topological group,
if $X$ is locally compact Hausdorff
and if~the $G$-action on $X$ is continuous,
then, in Lemma 6.1 of~\cite{adams:locfreeacts},
we~observe that the $G$-action on $X$ is fixpoint rare iff,
for every nonempty open $G$-invariant subset~$U$ of $X$,
the $G$-action on $U$ is~effective.
We say the $G$-action on $X$ is {\bf component effective}
if, for all~$g\in G\backslash\{1_G\}$,
for any connected component $X^\circ$ of $X$,
there exists $x\in X^\circ$ such that $gx\ne x$.
If $X$ is locally connected ({\it e.g.}, if $X$ is a manifold),
then each connected component of $X$ is an open subset of~$X$,
and so, in this setting, fixpoint rare implies component effective.
In \lref{lem-real-analytic-component-eff-iff-fixpt-rare} below,
we show: for $C^\omega$ actions,
fixpoint rare is equivalent to component effective.
In \cref{cor-cpt-gp-comp-eff-equal-fxptrare} below,
we show: for a $C^\infty$ action of a compact Lie group,
fixpoint rare is equivalent to component effective.

Let a topological group $G$ act on a set $X$.
We say the $G$-action on~$X$ is {\bf locally free} if,
for every $x\in X$, $\Stab_G(x)$ is discrete.

By an {\bf$\R$-variety} (resp.~$\C$-variety), we mean a
reduced separated scheme of~finite type over $\spec\,\R$ (resp.~$\spec\C$).
Keep in mind that, in~our usage in this paper, a variety may be reducible.

Let $\bfv$ be an $\R$-variety.
The set of~$\R$-points of~$\bfv$ is denoted $\bfv(\R)$.
For any $x\in\bfv(\R)$,
we identify $x:\spec\,\R\to\bfv$ with its image in~$\bfv$;
then $x$ is a Zariski closed point in $\bfv$.
We give to $\bfv(\R)$ the Hausdorff topology induced by
the standard Hausdorff topology on $\R$.
Let $\Gamma$~be the Galois group of $\C$ over $\R$.
Let $\bfv^\C$ be the complexification of~$\bfv$.
Then $\bfv^\C$ is a $\C$-variety
and there is a~canonical action of $\Gamma$ on~$\bfv^\C$.
The points of~$\bfv$ are in~1-1 corrspondence with~$\Gamma$-orbits in~$\bfv^\C$,
and the Zariski closed points of $\bfv$ are in~1-1 correspondence with
$\Gamma$-orbits of~Zariski closed points in~$\bfv^\C$.
Also, the $\R$-points of $\bfv$ are in~1-1 correspondence
with Zariski closed $\Gamma$-fixpoints in $\bfv^\C$.
Let $\Irr(\bfv)$ denote the set of~all irreducible components of $\bfv$.
Then $\Irr(\bfv)$ is in~1-1 correspondence with~the set of orbits
of the $\Gamma$-action on the set of~irreducible components of $\bfv^\C$.
For all $x\in\bfv$, we make the following definitions:
\begin{itemize}
\item the dimension of the Zariski closure of $\{x\}$ in $\bfv$ is denoted $d_x^\bfv$,
\item the set of irreducible components of $\bfv$ passing through $x$
will be~denoted $\Irr_x(\bfv):=\{\bfc\in\Irr(\bfv)\,|\,x\in\bfc\}$,
\item by $\dim_x\bfv:=\max\{(\dim\bfc)-d_x^\bfv\,|\,\bfc\in\Irr_x(\bfv)\}$,
we denote the~Krull dimension of the local ring at $x$ \qquad\quad and
\item the tangent space to $\bfv$ at $x$ will be denoted $T_x\bfv$;
then $T_x\bfv$ is a finite dimensional vector space over $\R$.
\end{itemize}
For all $x\in\bfv$, we have $\dim_x\bfv\le\dim T_x\bfv$.
We define
$$\Sm(\bfv)\quad:=\quad\{\,\,x\in\bfv\,\,|\,\,\dim_x\bfv=\dim T_x\bfv\,\,\}.$$
The set $\Sm(\bfv)$ is Zariski dense and Zariski open in~$\bfv$.
For any $x\in\bfv$, by $x$ is {\bf smooth in $\bfv$}, we mean: $x\in\Sm(\bfv)$,
{\it i.e.}, $\dim_x\bfv=\dim T_x\bfv$.
We say $\bfv$ is {\bf smooth} if $\bfv=\Sm(\bfv)$.

Let $M$ be a manifold.
We will denote the tangent bundle of $M$ by~$TM$;
a $C^\infty$ action of a Lie group $G$ on $M$ induces, by differentiation, a~$G$-action on $TM$.
For any $x\in M$, the tangent space to $M$ at~$x$ is~denoted $T_xM$.
For any $k\in\N_0$,
let $\scrf^kM$ denote the $k$th order frame bundle of $M$;
a $C^\infty$ action of a Lie group $G$ on $M$ induces, by~differentiation, a.k.a.~``prolongation'',
a $G$-action on~$\scrf^kM$.

\section{Real analytic actions\wrlab{sect-real-analytic-actions}}

\begin{lem}\wrlab{lem-real-analytic-component-eff-iff-fixpt-rare}
Let a Lie group $G$ have a $C^\omega$ action on a manifold~$M$.
Then the $G$-action on $M$ is component effective iff it is fixpoint rare.
\end{lem}

\begin{proof}
Since $M$ is a manifold, $M$ is locally connected, and ``if'' follows.
We wish to prove ``only if''.
Assume the $G$-action on $M$ is component effective.
We wish to show: the $G$-action on $M$ is fixpoint rare.
Let $g\in G$ be given, let a nonempty open subset $U$ of $M$ be given,
and assume, for all $x\in U$, that $gx=x$.
We wish to show that $g=1_G$.

For all $S\subseteq M$, let $\scri(S)$ be the interior of $S$ in $M$
and let $\scrc(S)$ be the closure of $S$ in $M$.
Since $U$ is open in $M$, we have $\scri(U)=U$.
Let $F:=\{x\in M\,|\,gx=x\}$.
Then $U\subseteq F$.
Give to~$M$ a $C^\omega$ manifold structure compatible with its $C^\infty$ manifold structure
s.t.~the $G$-action on $M$ is $C^\omega$.
Define $\phi,\iota\in C^\infty(M,M)$ by~$\phi(x)=gx$ and $\iota(x)=x$.
Let $d:=\dim M$.
For any open ball $B$ in $\R^d$,
for any~two power series converging at every point of $B$,
if they agree on a nonempty open subset of $B$,
then they agree on all of $B$.
So, since $F=\{x\in M\,|\,\phi(x)=\iota(x)\}$,
we see, for any open subset $V$ of $M$,
that: $[V\subseteq F]\,\Rightarrow\,[\scrc(V)\subseteq\scri(F)]$.
Let $V_0:=\scri(F)$.
Then $\scrc(V_0)\subseteq\scri(F)$.
Also, $V_0$ is open in $M$.
Since $\scrc(V_0)\subseteq\scri(F)=V_0\subseteq\scrc(V_0)$,
we get: $\scrc(V_0)=V_0$,
so $V_0$~is closed in~$M$.
As $\emptyset\ne U=\scri(U)\subseteq\scri(F)=V_0$,
we get: $\emptyset\ne V_0$.
So, since $V_0$ is open in $M$ and closed in~$M$,
choose a connected component $C$ of~$M$ such that $C\subseteq V_0$.
Then we have $C\subseteq V_0=\scri(F)\subseteq F$.
Then $\forall x\in C,\,gx=x$.
Then, as the $G$-action on~$M$ is component effective,
we get $g=1_G$.
\end{proof}

\section{Actions of compact Lie groups\wrlab{sect-act-cpt-gp}}

Let a Lie group $K$ have a $C^\infty$ action on a manifold $M$.

\begin{lem}\wrlab{lem-precptelt-ordone-fixpt-fix-component}
Assume $K$ is compact.
Let $g\in K$, let $p\in M$ and assume,
for all $v\in T_pM$, that $gv=v$.
Let $M^\circ$ be the connected component of $M$ satisfying $p\in M^\circ$.
Then, for all $q\in M^\circ$, we have $gq=q$.
\end{lem}

\begin{proof}
Let $F:=\{q\in M\,|\,gq=q\}$.
We wish to show that $M^\circ\subseteq F$.

Let $F':=\{q\in M\,|\,\forall v\in T_qM,\,gv=v\}$.
Then $F'\subseteq F$, so it suffices to show that $M^\circ\subseteq F'$.
By assumption, $p\in F'$.
Then, by definition of~$M^\circ$,
it suffices to show that $F'$ is both open in $M$ and closed in $M$.

Let $S:=\{v\in TM\,|\,gv\ne v\}$.
Let $\pi:TM\to M$ be the tangent bundle map.
As $S$ is open in $TM$ and $\pi:TM\to M$ is an~open map,
$\pi(S)$ is open in $M$.
So, as $F'=M\backslash[\pi(S)]$, $F'$ is closed in $M$.
It remains to show: $F'$ is open in~$M$.
Let $q\in F'$ be given.
We wish to~show: there exists an open neighborhood $U$ of $q$ in~$M$
such that~$U\subseteq F'$.

As $K$~is compact,
fix a $K$-invariant Riemannian metric on $M$.
Let $\varepsilon:TM\dashrightarrow M$ be the Riemannian exponential map.
Choose an open neighborhood $U$ of $q$ in $M$ s.t.~$U\subseteq\varepsilon_*(T_qM)$.
We wish to show:~$U\subseteq F'$.

As $q\in F'$, we see, for all~$v\in T_qM$, that $gv=v$.
By naturality, $\varepsilon:TM\dashrightarrow M$ is $K$-equivariant.
Then, for all $x\in\varepsilon_*(T_qM)$, $gx=x$.
So, since $U\subseteq\varepsilon_*(T_qM)$,
we see, for all $x\in U$, that $gx=x$.
So, since $U$~is open in $M$
and since the $K$-action on $M$ is $C^\infty$,
we see, for all~$x\in U$, for all $v\in T_xM$, that $gv=v$.
That is, $U\subseteq F'$, as desired.
\end{proof}

\begin{cor}\wrlab{cor-cpt-gp-comp-eff-equal-fxptrare}
Assume $K$ is compact.
Then the $K$-action on~$M$ is component effective iff it is fixpoint rare.
\end{cor}

\begin{proof}
Since $M$ is a manifold, $M$ is locally connected, and ``if'' follows.
We wish to prove ``only if''.
Assume the $K$-action on $M$ is component effective.
We wish to show: the $K$-action on $M$ is fixpoint rare.
Let $g\in K$ be given, let a nonempty open subset $U$ of $M$ be given,
and assume, for all $x\in U$, that $gx=x$.
We wish to show that $g=1_K$.

As $U$ is open in $M$ and the $K$-action on $M$ is $C^\infty$
and $\forall x\in U,\,gx=x$,
we get: $\forall x\in U,\,\forall v\in T_xM,\,gv=v$.
As $U\ne\emptyset$, choose $p\in U$.
Then: $\forall v\in T_pM,\,gv=v$.
Let $M^\circ$ be the connected component of $M$ satisfying $p\in M^\circ$.
Then, by~\lref{lem-precptelt-ordone-fixpt-fix-component}, for all $q\in M^\circ$, we have $gq=q$.
So, since the $K$-action on $M$ is component effective, $g=1_K$.
\end{proof}

\section{Stabilizers lack compact subgroups\wrlab{sect-stab-lack-cpt}}

Let a Lie group $G$ have a $C^\infty$ action on a manifold $M$.

\begin{lem}\wrlab{lem-no-cpt-subgps-in-stabs}
Assume the $G$-action on $M$ is component effective.
Let $\ell\in\N$, $x\in\scrf^\ell M$.
Then $\Stab_G(x)$ has no nontrivial compact subgroups.
\end{lem}

\begin{proof}
Let a compact subgroup $K$ of $\Stab_G(x)$ be given.
We wish to~show: $K=\{1_G\}$.
Let $g\in K$ be given.
We wish to show: $g=1_G$.

Let $\pi:\scrf^\ell M\to M$ and $\pi_1:\scrf^1M\to M$ be
the structure maps of~$\scrf^\ell M$ and $\scrf^1M$, respectively.
Let $\tau:\scrf^\ell M\to\scrf^1M$ be the canonical map.
Then $\pi=\pi_1\circ\tau$.
Let $p:=\pi(x)$ and $x_1:=\tau(x)$.
Then $\pi_1(x_1)=p$.
As $\tau:\scrf^\ell M\to\scrf^1M$ is $G$-equivariant,
$\Stab_G(x)\subseteq\Stab_G(x_1)$.
Let $M^\circ$~be the connected component of $M$ satisfying $p\in M^\circ$.
The $G$-action on $M$ is component effective, so it suffices to show:
$\forall q\in M^\circ,\,gq=q$.
Then, by \lref{lem-precptelt-ordone-fixpt-fix-component},
it suffices to show: $\forall v\in T_pM,\,gv=v$.
Let $v\in T_pM$ be given.
We wish to prove that $g\in\Stab_G(v)$.

By Lemma 3.1 of \cite{adams:genfreeacts}
(with $k=1$, with $\sigma$ replaced by $\pi_1$, $x$ by $x_1$, $w$~by~$p$,
and with $W=\scrf^0M$ identified with $M$), $\Stab_G(x_1)\subseteq\Stab_G(v)$.
Then $g\in K\subseteq\Stab_G(x)\subseteq\Stab_G(x_1)\subseteq\Stab_G(v)$, as desired.
\end{proof}

\begin{cor}\wrlab{cor-cpt-stabs-in-subset-prolong}
Assume the $G$-action on $M$ is component effective.
Let~$\ell\in\N$.
Let $V\subseteq\scrf^\ell M$ be $G$-invariant.
Assume, for all $x\in V$, that $\Stab_G(x)$ is compact.
Then the $G$-action on $V$ is free.
\end{cor}

\begin{proof}
Given $x\in V$.
Let $S:=\Stab_G(x)$.
We wish to show: $S=\{1_G\}$.

By assumption, $S$ is compact.
By \lref{lem-no-cpt-subgps-in-stabs},
$S$ has no nontrival compact subgroups.
Then $S$ is trivial, {\it i.e.}, $S=\{1_G\}$, as desired.
\end{proof}

\begin{cor}\wrlab{cor-cpt-stabs-in-prolong}
Assume the $G$-action on $M$ is component effective.
Let~$\ell\in\N$.
Assume, for all $x\in\scrf^\ell M$, that $\Stab_G(x)$ is compact.
Then the $G$-action on $\scrf^\ell M$ is free.
\end{cor}

\begin{proof}
Let $V:=\scrf^\ell M$, and apply \cref{cor-cpt-stabs-in-subset-prolong}.
\end{proof}

\begin{cor}\wrlab{cor-cpt-stabs-prolong}
Assume the $G$-action on $M$ is component effective.
Assume, for all $q\in M$, that $\Stab_G(q)$ is compact.
Let $\ell\in\N$.
Then the $G$-action on $\scrf^\ell M$ is free.
\end{cor}

\begin{proof}
By \cref{cor-cpt-stabs-in-prolong},
it suffices to show, for all $x\in\scrf^\ell M$,
that $\Stab_G(x)$ is compact.
Let $x\in\scrf^\ell M$ be given and let $S:=\Stab_G(x)$.
Then $S$ is closed in $G$, and we wish to show that $S$ is compact.

Let $\pi:\scrf^\ell M\to M$ be the structure map of~$\scrf^\ell M$,
and let $q:=\pi(x)$.
The mapping $\pi:\scrf^\ell M\to M$ is $G$-equivariant,
so $\Stab_G(x)\subseteq\Stab_G(q)$.
Let $C:=\Stab_G(q)$.
By assumption, $C$~is compact.
Also, we have $S=\Stab_G(x)\subseteq\Stab_G(q)=C$.
As $S$ is closed in $G$ and $S\subseteq C\subseteq G$, we get: $S$ is closed in $C$.
So, since $C$~is compact, $S$ is compact.
\end{proof}

\begin{cor}\wrlab{cor-proper-prolong-strong}
Assume the $G$-action on $M$ is component effective.
Let~$\ell\in\N$.
Assume the $G$-action on $\scrf^\ell M$ is proper.
Then the $G$-action on~$\scrf^\ell M$ is free.
\end{cor}

\begin{proof}
The $G$-action on $\scrf^\ell M$ is proper,
so, for all $x\in\scrf^\ell M$, $\Stab_G(x)$~is compact.
Then, by \cref{cor-cpt-stabs-in-prolong},
the $G$-action on $\scrf^\ell M$ is free.
\end{proof}

\begin{cor}\wrlab{cor-prolong-proper}
Assume that the $G$-action on $M$ is both proper and component effective.
Let $\ell\in\N$.
Then the $G$-action on~$\scrf^\ell M$ is both proper and free.
\end{cor}

\begin{proof}
The structure map $\scrf^\ell M\to M$ is $G$-equivariant and continuous,
so, since $G$-action on $M$ is proper,
it follows that the $G$-action on $\scrf^\ell M$ is proper.
Then, by \cref{cor-proper-prolong-strong},
the $G$-action on $\scrf^\ell M$ is free.
\end{proof}

\section{Some algebraic geometry over $\R$\wrlab{sect-some-ag-over-r}}

For any $\R$-variety $\bfv$, we will denote the Zariski closure of~$\bfv(\R)$ in $\bfv$ by $\widecheck{\bfv}$.
Let $\scrv$ be the category of $\R$-varieties.
Let $\widecheck{\scrv}$ be the full subcategory
$\{\bfv\in\scrv\,|\,\widecheck{\bfv}=\bfv\}$ of $\scrv$.
In this section, we investigate $\widecheck{\scrv}$.

Let $\bfv$ be an $\R$-variety.
For any Zariski locally closed $\bfs\subseteq\bfv$, we give to $\bfs$ its natural $\R$-variety structure;
then the Zariski topology on~$\bfs$ is equal to the relative topology inherited from the Zariski topology on~$\bfv$.
For all~$x\in\Sm(\bfv)$,
the local ring of $\bfv$ at~$x$ is regular;
a regular local ring is an integral domain, so $\#[\Irr_x(\bfv)]=1$.
We conclude that no~smooth point of~$\bfv$ lies on two irreducible components of~$\bfv$.
Let
$$\bfv_1\quad:=\quad\{\,\,x\in\bfv\,\,|\,\,\#[\Irr_x(\bfv)]=1\,\,\}.$$
Then $\ds{\Sm(\bfv)=\bfv_1\cap\left[\bigcup\{\Sm(\bfc)\,|\,\bfc\in\Irr(\bfv)\}\right]}$.
That is, a point of~$\bfv$ is smooth in $\bfv$ iff it both lies on exactly one irreducible component of~$\bfv$
and is a smooth point of that irreducible component.

Let $n\in\N$.
We denote, by $\bfa^n$, affine $n$-space over $\R$;
then 
\begin{itemize}
\item$\bfa^n$ is an irreducible $\R$-variety,
\item$\bfa^n(\R)$ is naturally identified with $\R^n$ \qquad\quad and
\item$\bfa^n(\R)$ is Zariski dense in $\bfa^n$, {\it i.e.}, $\bfa^n\in\widecheck{\scrv}$.
\end{itemize}
We denote, by~$P_n$, the $\R$-algebra of~$\R$-morphisms $\bfa^n\to\bfa^1$.
For any~$S\subseteq\bfa^n$,
we set
$\scri(S):=\{\bflitf\in P_n\,|\,\forall s\in S,\,\bflitf(s)=0\}$.
For all $\bflitf\in P_n$, for all~$x\in\bfa^n$,
the linearization at $x$ of $\bflitf$ is denoted
$d_x\bflitf:\R^n\to\R$.
For any Zariski closed $\R$-subvariety $\bfv$ of $\bfa^n$, for any $x\in\bfv$,
we identify~the tangent space $T_x\bfv$
with $\ds{\bigcap\,\{\,\ker\,(d_x\bflitf)\,|\,\bflitf\in\scri(\bfv)\}}$.

Let $n\in\N$ and let $f\in C^\infty(\R^n,\R)$ and let $x\in\R^n$.
Then the linearization at $x$ of $f$ will be denoted $d_xf:\R^n\to\R$.
For any $\bflitf\in P_n$, if~$f:\R^n\to\R$ is the $\R$-points of~$\bflitf:\bfa^n\to\bfa^1$,
then $d_xf=d_x\bflitf$.

Let $n,d\in\N$.
By an {\bf$(n,d)$-variety}, we mean a
Zariski locally closed $\R$-subvariety~$\bfv$ of $\bfa^n$ such that,
for all $\bfc\in\Irr(\bfv)$, $\dim\bfc<d$.

Let $M$ be a manifold, let $n:=\dim M$, let $S\subseteq M$ and let $x\in S$.
Let $d\in\{0,\ldots,n\}$.
Let $c:=n-d$.
Then an {\bf$(M,S,d)$-chart near $x$} is
a partial function $\phi:M\dashrightarrow\R^n$ such that
\begin{itemize}
\item[]( $\dom[\phi]$ is an open neighborhood of $x$ in $M$ ) \qquad and
\item[]( $\im[\phi]=\R^n$ ) \qquad and \qquad ( $\phi_*(S)=\R^d\times\{0_c\}$ ) \qquad and
\item[]( $\phi:\dom[\phi]\to\R^n$ is a diffeomorphism ).
\end{itemize}
By {\bf$S$ is $d$-dimensional in $M$ near $x$}, we mean
that there exists an~$(M,S,d)$-chart near $x$.
If $S$ is $d$-dimensional in $M$ near $x$,
then, for~any~neighborhood $U$ of $x$ in $M$,
there exists an $(M,S,d)$-chart $\phi$ near~$x$
such that $\dom[\phi]\subseteq U$.

Let $M$ be a manifold, let $n:=\dim M$ and let $d\in\{0,\ldots,n\}$.
Let $S\subseteq M$.
Then {\bf$S$~is a locally closed $d$-submanifold of $M$} means:
both (~$S\ne\emptyset$~) and 
( for all $x\in S$, $S$ is $d$-dimensional in $M$ near $x$ ).

Let $M$ be a manifold, let $n:=\dim M$ and let $S\subseteq M$.
Then, by {\bf$S$ is a locally closed submanifold of $M$}, we mean:
there exists $d\in\{0,\ldots,n\}$ such that $S$ is a locally closed $d$-submanifold of $M$.

Let $M$ be a manifold and let $S$ be a locally closed submanifold of~$M$.
Then $S$ is a locally closed subset of $M$,
{\it i.e.}, $S$ is the intersection of~an~open subset of $M$ and a closed subset of $M$.
We give to~$S$ the unique manifold structure such that the inclusion map $S\to M$ is an immersion.
Then the manifold topology on $S$ agrees with the relative topology on $S$ inherited from $M$.

Let $n,d\in\N$.
Let $V\subseteq\R^n$.
By {\bf$V$ is an $(n,d)$-assembly}, we mean: either ( $V=\emptyset$ ) or ( there exist $k\in\N$
and pairwise-disjoint locally closed submanifolds $V_1,\ldots,V_k$ of $\R^n$
such that both $[\,V=V_1\cup\cdots\cup V_k\,]$ and $[\,\forall j\in\{1,\ldots,k\},\,\dim V_j<d\,]$ ).

\begin{lem}\wrlab{lem-1-Rpts-irred-submfld}
Let $n\in\N$.
Let $\bfs$ be a Zariski locally closed $\R$-subvariety of~$\bfa^n$.
Assume $\bfs$ is irreducible and smooth.
Let $S:=\bfs(\R)$, $d:=\dim\bfs$.
Assume $S\ne\emptyset$.
Then $S$ is a locally closed $d$-submanifold of $\R^n$.
\end{lem}

\begin{proof}
Given $x\in S$,
we wish to show: $S$ is $d$-dimensional in~$\R^n$ near $x$.

Let $\bfv$ be the Zariski closure of $\bfs$ in $\bfa^n$.
Then $x\in\bfs\subseteq\bfv\subseteq\bfa^n$.
Because $\bfs$~is irreducible, $\bfv$ is also irreducible.
Because $\bfs$ is Zariski locally closed in $\bfa^n$,
$\bfs$ is Zariski open in $\bfv$.
It follows both that $\dim\bfs=\dim\bfv$
and that $\dim T_x\bfs=\dim T_x\bfv$.
As $\bfs$ is smooth, we get $\dim T_x\bfs=\dim_x\bfs$.
Since $\bfs$ is irreducible,
we get $\dim_x\bfs=\dim\bfs$.
Then $\dim T_x\bfs=\dim_x\bfs=\dim\bfs=d$.
Then we have both
$$\dim\bfv\,=\,\dim\bfs\,=\,d\qquad\hbox{and}\qquad\dim T_x\bfv\,=\,\dim T_x\bfs\,=\,d.$$
Let $\bfy:=\bfv\backslash\bfs$.
Since $\bfs$ is Zariski open in $\bfv$,
it follows that $\bfy$ is Zariski closed in $\bfv$.
So, since $\bfv$ is Zariski closed in $\bfa^n$,
we see that $\bfy$~is Zariski closed in~$\bfa^n$.
Also, $\bfv\backslash\bfy=\bfs$.
Since $x\in\bfs$, $x\notin\bfy$.
Let $\Lambda$ denote the dual of $\R^n$;
that is, $\Lambda$ is the vector space of linear maps $\R^n\to\R$.
Let $\Lambda_1:=\{d_x\bflitf\,|\,\bflitf\in\scri(\bfv)\}$.
Then $\Lambda_1$ is a vector subspace of $\Lambda$.
We identify $T_x\bfv$ with $\displaystyle{\bigcap\{\ker\,\lambda\,|\,\lambda\in\Lambda_1\}}$.
Then, as $\dim T_x\bfv=d$, we have $\dim\Lambda_1=n-d$.
Let $c:=n-d$ and choose $\bflitf_1,\ldots,\bflitf_c\in\scri(\bfv)$
such that $d_x\bflitf_1,\ldots,d_x\bflitf_c$ are linearly independent in $\Lambda$.
Let $\bfw:=\{p\in\bfa^n\,|\,\bflitf_1(p)=\cdots=\bflitf_c(p)=0\}$.
Then~$\bfw$~is Zariski closed in~$\bfa^n$ and $\bfv\subseteq\bfw$
and $\bflitf_1,\ldots,\bflitf_c\in\scri(\bfw)$.
Let
$$K\quad:=\quad[\ker\,(d_x\bflitf_1)]\,\,\,\cap\,\,\,\cdots\,\,\,\cap\,\,\,[\ker\,(d_x\bflitf_c)].$$
Because $d_x\bflitf_1,\ldots,d_x\bflitf_c$ are linearly independent in $\Lambda$,
$\dim K=n-c=d$.
We identify $T_x\bfw$ with
$\displaystyle{\bigcap\{\ker(d_x\bflitf)\,|\,\bflitf\in\scri(\bfw)\}}$.
As $\bflitf_1,\ldots,\bflitf_c\in\scri(\bfw)$,
we get $T_x\bfw\subseteq K$.
Then $\dim T_x\bfw\le\dim K=d$.
Since $\bfv$ is irreducible and since $\bfv\subseteq\bfw$,
choose $\bfx_1\in\Irr(\bfw)$ such that $\bfv\subseteq\bfx_1$.
Then
$$x\quad\in\quad\bfs\quad\subseteq\quad\bfv\quad\subseteq\quad\bfx_1\quad\subseteq\quad\bfw\quad\subseteq\quad\bfa^n.$$
Since $\bfv$ is Zariski closed in $\bfa^n$ and since $\bfv\subseteq\bfx_1\subseteq\bfa^n$,
we see that $\bfv$ is Zariski closed in $\bfx_1$.
Since $\bfx_1\subseteq\bfw$, $T_x\bfx_1\subseteq T_x\bfw$,
and so $\dim T_x\bfx_1\le\dim T_x\bfw$.
As $\bfx_1$~is irreducible, $\dim_x\bfx_1=\dim\bfx_1$.
Then
\begin{eqnarray*}
\dim\bfx_1&=&\dim_x\bfx_1\,\,\le\,\,\dim T_x\bfx_1\\
&\le&\dim T_x\bfw\,\,\le\,\,d\,\,=\,\,\dim\bfv.
\end{eqnarray*}
So, as $\bfv$ and $\bfx_1$ are irreducible varieties and
as $\bfv$ is Zariski closed in~$\bfx_1$,
we get $\bfv=\bfx_1$.
We have $\bfs\subseteq\bfw$,
so $\dim_x\bfs\le\dim_x\bfw$.
Recall: $d=\dim\bfs$ and $\dim_x\bfs=\dim\bfs$ and $\dim T_x\bfw\le d$.
Then
$$d\,\,\,=\,\,\,\dim\bfs\,\,\,=\,\,\,\dim_x\bfs\,\,\,\le\,\,\,\dim_x\bfw\,\,\,\le\,\,\,\dim T_x\bfw\,\,\,\le\,\,\,d.$$
Then $\dim_x\bfw=\dim T_x\bfw$, so $x\in\Sm(\bfw)$,
so $\#[\Irr_x(\bfw)]=1$.
So, since~$x\in\bfx_1\in\Irr(\bfw)$, we see that $\Irr_x(\bfw)=\{\bfx_1\}$.
We define $k:=\#[\Irr(\bfw)]$
and let $\bfx_2,\ldots,\bfx_k$ be the distinct elements of~$[\Irr(\bfw)]\backslash\{\bfx_1\}$.
Then $\Irr(\bfw)=\{\bfx_1,\ldots,\bfx_k\}$,
and it follows that $\bfw=\bfx_1\cup\cdots\cup\bfx_k$.
For all $j\in\{2,\ldots,k\}$, $\bfx_j\in(\Irr(\bfw))\backslash(\Irr_x(\bfw))$,
so $x\notin\bfx_j$.
Let $\bfz:=\bfx_2\cup\cdots\cup\bfx_k$.
Then $x\notin\bfz$.
Let $\bfu:=\bfa^n\backslash(\bfy\cup\bfz)$.
Since $\bfx_2,\ldots,\bfx_k\in\Irr(\bfw)$,
it follows that $\bfx_2,\ldots,\bfx_k$ are Zariski closed in $\bfw$,
and so $\bfz$ is Zariski closed in $\bfw$.
So, since $\bfw$ is Zariski closed in $\bfa^n$,
we conclude that $\bfz$ is Zariski closed in $\bfa^n$.
So, since $\bfy$~is also Zariski closed in $\bfa^n$,
it follows that $\bfu$ is Zariski open in $\bfa^n$.
So, since $x\notin\bfy$ and $x\notin\bfz$,
we see that $\bfu$ is a Zariski open neighborhood of $x$ in~$\bfa^n$.
Also, we have $\bfu\cap\bfw=\bfw\backslash(\bfy\cup\bfz)$.
Since $\bfs\subseteq\bfw$,
we get $\bfu\cap\bfs\subseteq\bfu\cap\bfw$.
Recall that $\bfv\backslash\bfy=\bfs$.
We have
$$\bfw\backslash\bfz\,\,\,=\,\,\,(\bfx_1\cup\cdots\cup\bfx_k)\backslash(\bfx_2\cup\cdots\cup\bfx_k)\,\,\,\subseteq\,\,\,\bfx_1\,\,\,=\,\,\,\bfv,$$
so $\bfw\backslash(\bfy\cup\bfz)\subseteq\bfv\backslash\bfy$.
Then $\bfu\cap\bfw=\bfw\backslash(\bfy\cup\bfz)\subseteq\bfv\backslash\bfy=\bfs$.
So, since $\bfu\cap\bfw\subseteq\bfu$,
we get $\bfu\cap\bfw\subseteq\bfu\cap\bfs$.
So, since $\bfu\cap\bfs\subseteq\bfu\cap\bfw$,
we get $\bfu\cap\bfs=\bfu\cap\bfw$.
Let $U:=\bfu(\R)$ and let $W:=\bfw(\R)$.
Then $U$ is an open neighborhood of $x$ in $\R^n$
and $U\cap S=U\cap W$.
For all~$j\in\{1,\ldots,c\}$, let $f_j:\R^n\to\R$
be the $\R$-points of $\bflitf_j:\bfa^n\to\bfa^1$;
then $d_xf_j=d_x\bflitf_j$.
Then $d_xf_1,\ldots,d_xf_c$ are linearly independent in $\Lambda$.
So, as $x\in W=\{p\in\R^n\,|\,f_1(p)=\cdots=f_c(p)=0\}$ and $n-c=d$,
we see, by the Implicit Function Theorem, that $W$ is $d$-dimensional near $x$.
Choose an $(\R^n,W,d)$-chart $\phi$ near $x$ such that $\dom[\phi]\subseteq U$.
Because $\phi_*(S)=\phi_*(U\cap S)=\phi_*(U\cap W)=\phi_*(W)$,
we conclude that $\phi$ is an~$(\R^n,S,d)$-chart near~$x$.
Then $S$ is $d$-dimensional in~$\R^n$ near~$x$.
\end{proof}

\begin{lem}\wrlab{lem-2-Rpts-smooth-assembly}
Let $n,d\in\N$.
Let $\bfs$ be a smooth $(n,d)$-variety
and let $S:=\bfs(\R)$.
Then $S$ is an $(n,d)$-assembly.
\end{lem}

\begin{proof}
Let $k:=\#[\Irr(\bfs)]$.
Let $\bfs_1,\ldots,\bfs_k$ be the distinct elements of~$\Irr(\bfs)$.
Then, since $\bfs$ is smooth, $\bfs_1,\ldots,\bfs_k$ are pairwise-disjoint.
For all~$j\in\{1,\ldots,k\}$, let~$d_j:=\dim\bfs_j$;
by definition of~$(n,d)$-variety, we get $d_j<d$.
For all $j\in\{1,\ldots,k\}$, let $S_j:=\bfs_j(\R)$;
then, by~\lref{lem-1-Rpts-irred-submfld},
either (~$S_j=\emptyset$ ) or ( $S_j$~is a locally closed closed $d_j$-submanifold of~$\R^n$~).
Since ( $\bfs_1,\ldots,\bfs_k$ are pairwise-disjoint ) and since ( $\bfs=\bfs_1\cup\cdots\cup\bfs_k$~),
we get: ( $S_1,\ldots,S_k$ are pairwise-disjoint ) and (~$S=S_1\cup\cdots\cup S_k$ ).
Then $S$ is an $(n,d)$-assembly, as desired.
\end{proof}

\begin{lem}\wrlab{lem-3-Rpts-variety-assembly}
Let $n,d\in\N$, and let $\bfv$ be an $(n,d)$-variety.
Define $V:=\bfv(\R)$.
Then $V$ is an $(n,d)$-assembly.
\end{lem}

\begin{proof}
If $d=1$, then,
by Proposition 1.6 (p.~8) of \cite{coste:realalgsets},
$\bfv$ is finite;
in this case, $V$ is a finite subset of $\R^n$,
and so $V$ is an $(n,d)$-assembly.
We assume $d>1$, and argue by induction on $d$.

Let $\bfs:=\Sm(\bfv)$ and $\bfz:=\bfv\backslash\bfs$.
Then $\bfs$ is Zariski open in $\bfv$, and so $\bfz$ is Zariski closed in $\bfv$.
Also, $\bfs$ is Zariski dense in~$\bfv$.
Let $S:=\bfs(\R)$ and $Z:=\bfz(\R)$.
Since $\bfv$ is an $(n,d)$-variety
and since $\bfs$ is Zariski open in~$\bfv$,
we conclude that $\bfs$ is an~$(n,d)$-variety,
so, by \lref{lem-2-Rpts-smooth-assembly},
$S$~is an~$(n,d)$-assembly.
Since $\bfs\cap\bfz=\emptyset$ and $\bfv=\bfs\cup\bfz$,
we see that $S\cap Z=\emptyset$ and $V=S\cup Z$.
It therefore suffices to~show: $Z$~is an $(n,d)$-assembly.
Any $(n,d-1)$-assembly is an $(n,d)$-assembly,
so, it suffices to show: $Z$~is an $(n,d-1)$-assembly.
Then, by~the induction hypothesis, it suffices to show: $\bfz$ is an $(n,d-1)$-variety.

By definition of $(n,d)$-variety, $\bfv$ is Zariski locally closed in $\bfa^n$.
So, as $\bfz$ is Zariski closed in $\bfv$,
we see that $\bfz$~is Zariski locally closed in $\bfa^n$.
It remains to show: each irreducible component of $\bfz$ has $\dim<d-1$.
Let $\bfb\in\Irr(\bfz)$ be given.
We wish to prove: $\dim\bfb<d-1$.

Since $\bfb\in\Irr(\bfz)$, it follows that $\bfb$ is Zariski closed in $\bfz$.
So, since $\bfz$ is Zariski closed in $\bfv$,
we conclude that $\bfb$ is Zariski closed in $\bfv$.
Since $\bfb\subseteq\bfz\subseteq\bfv$ and since $\bfb$ is irreducible,
choose $\bfc\in\Irr(\bfv)$ such that $\bfb\subseteq\bfc$.
Since $\bfb$ is Zariski closed in $\bfv$
and since $\bfb\subseteq\bfc\subseteq\bfv$,
it follows that $\bfb$ is Zariski closed in $\bfc$.
Since $\bfc\in\Irr(\bfv)$,
we see that the Zariski interior of~$\bfc$ in~$\bfv$ is nonempty.
So, since $\bfs$ is Zariski dense in~$\bfv$,
$\bfc\cap\bfs\ne\emptyset$.
Then $\bfc\not\subseteq\bfv\backslash\bfs$.
So, since $\bfb\subseteq\bfz=\bfv\backslash\bfs$,
we get $\bfb\ne\bfc$.
So, since $\bfb$ and $\bfc$ are both irreducible
and $\bfb$ is Zariski closed in $\bfc$,
we get $\dim\bfb\le(\dim\bfc)-1$.
Since $\bfc\in\Irr(\bfv)$
and $\bfv$~is an $(n,d)$-variety,
we see that $\dim\bfc<d$.
Then $\dim\bfb\le(\dim\bfc)-1<d-1$.
\end{proof}

\begin{lem}\wrlab{lem-4-submflds-nwh-dense}
Let $M$ be a manifold, $X$ a locally closed submanifold of $M$.
Assume $\dim X<\dim M$.
Then $X$ is nowhere dense in~$M$.
\end{lem}

\begin{proof}
Let $U$ denote the interior of $X$ in $M$.
For locally closed sets, nowhere dense is equivalent to empty interior,
so it suffices to prove that $U=\emptyset$.
Assume that $U\ne\emptyset$.
We aim for a contradiction.

As $U\ne\emptyset$, choose $x\in U$.
Then $x\in U\subseteq X\subseteq M$.
Then $U$ is an open neighborhood of $x$ in $M$.
Let $d:=\dim X$.
Then $X$ is a locally closed $d$-submanifold of $M$,
so, as $x\in X$, $X$ is $d$-dimensional in~$M$ near $x$.
Choose an $(M,X,d)$-chart $\phi$ near~$x$ s.t.~$\dom[\phi]\subseteq U$.
Let $n:=\dim M$.
Let $c:=n-d$.
Let $Y:=\R^d\times\{0_c\}$.
Since $d=\dim X<\dim M=n$, we get $Y\subsetneq\R^n$.
By definition of $(M,X,d)$-chart, $\phi_*(X)=Y$ and $\im[\phi]=\R^n$.
Since $\dom[\phi]\subseteq U$, we see that $\phi_*(U)=\im[\phi]$.
Then $\R^n=\im[\phi]=\phi_*(U)\subseteq\phi_*(X)=Y\subsetneq\R^n$,
contradiction.
\end{proof}

\begin{lem}\wrlab{lem-5-ndassemb-not-ddim}
Let $n,d\in\N$.
Let $V$ be an $(n,d)$-assembly.
Let $x\in V$.
Then $V$ is not $d$-dimensional in $\R^n$ near $x$.
\end{lem}

\begin{proof}
Assume, for a contradiction, $V$ is $d$-dimensional in $\R^n$ near $x$.

Choose an $(\R^n,V,d)$-chart $\phi$ near $x$.
Let $c:=n-d$, $Y:=\R^d\times\{0_c\}$.
Then $Y$ is a locally closed submanifold of $\R^n$
and $\dim Y=d$.
Let $D:=\dom[\phi]$.
By definition of $(\R^n,V,d)$-chart,
we see: ( $D$ is an open subset of $\R^n$ )
and ( $\phi_*(V)=Y$ ) and (~$\im[\phi]=\R^n$ ).
Since $x\in V$, $V\ne\emptyset$.
Then $V$ is a nonempty $(n,d)$-assembly,
so choose $k\in\N$ and choose pairwise-disjoint locally closed submanifolds
$V_1,\ldots,V_k$ of $\R^n$ such that both
$[\,V=V_1\cup\cdots\cup V_k\,]$ and $[\,\forall j\in\{1,\ldots,k\},\,\dim V_j<d\,]$.
Let $I:=\{1,\ldots,k\}$.
For all $j\in I$, let $d_j:=\dim V_j$; then $d_j<d$.

For all $j\in I$, let $D_j:=D\cap V_j$;
then, as $D$ is open in $\R^n$
and $V_j$ is a locally closed $d_j$-submanifold of $\R^n$,
we get: either ( $D_j=\emptyset$ )
or ( $D_j$~is a locally closed $d_j$-submanifold of $D$ ).
For all $j\in I$, let $Y_j:=\phi_*(V_j)$;
then we have
$Y_j=\phi_*(V_j)=\phi_*((\dom[\phi])\cap V_j)=\phi_*(D\cap V_j)=\phi_*(D_j)$.
Also, $Y=\phi_*(V)=\phi_*(V_1\cup\cdots\cup V_k)=Y_1\cup\cdots\cup Y_k$.
Choose $j\in I$ such that $Y_j$ is not nowhere dense in $Y$.
Then $\phi_*(D_j)=Y_j\ne\emptyset$, and it follows that $D_j\ne\emptyset$.
Then $D_j$~is a locally closed $d_j$-submanifold of~$D$.
So, since $\phi_*(D_j)=Y_j$ and since $\phi:D\to\R^n$ is a diffeomorphism,
we conclude that $Y_j$ is a locally closed $d_j$-submanifold of~$\R^n$.
So, since $Y_j\subseteq Y$, it follows that $Y_j$ is a locally closed $d_j$-submanifold of $Y$.
Then, as $\dim Y_j=d_j<d=\dim Y$, by~\lref{lem-4-submflds-nwh-dense} (with $M$~replaced by $Y$ and $X$~by~$Y_j$),
$Y_j$ is nowhere dense in $Y$, contradicting our choice of~$j$.
\end{proof}

\begin{lem}\wrlab{lem-6-sm-Rpt-imples-Zdense-Rpts}
Let $\bfx$ be an irreducible $\R$-variety.
Let $\bfs:=\Sm(\bfx)$ and let $S:=\bfs(\R)$.
Assume $S\ne\emptyset$.
Then $S$ is Zariski dense in $\bfx$.
\end{lem}

\begin{proof}
As $S\ne\emptyset$, choose $s_0\in S$.
Choose a Zariski open neighborhood~$\bfx_0$ of $s_0$ in $\bfx$
such that $\bfx_0$ is affine.
Since $\bfx_0$ is nonempty and Zariski open in~$\bfx$ and
since $\bfx$~is irreducible,
we get:
( $\bfx_0$ is Zariski dense in~$\bfx$ ) and
( $\bfx_0$ is irreducible ).
Let $X_0:=\bfx_0(\R)$.
Then $s_0\in X_0$.
Let $\bfs_0:=\Sm(\bfx_0)$, $S_0:=\bfs_0(\R)$.
As $\bfx_0$ is Zariski open in~$\bfx$,
we get $[\Sm(\bfx)]\cap\bfx_0=\Sm(\bfx_0)$,
{\it i.e.}, $\bfs\cap\bfx_0=\bfs_0$.
Then $s_0\in S\cap X_0=S_0$.
Since $S_0\subseteq S$, it suffices to show: $S_0$~is Zariski dense in $\bfx$.
So, since $\bfx_0$ is Zariski dense in~$\bfx$,
it suffices to~show: $S_0$ is Zariski dense in~$\bfx_0$.

As $\bfx_0$ is affine,
choose $n\in\N$ and a Zariski closed $\R$-subvariety $\bfx_1$ of $\bfa^n$
such that $\bfx_0$ is $\R$-isomorphic to $\bfx_1$.
Let $\bfalpha:\bfx_0\to\bfx_1$ be an~$\R$-isomorphism.
Since $\bfx_0$ is irreducible, we conclude that $\bfx_1$ is also irreducible.
Let $X_1:=\bfx_1(\R)$.
Let $\alpha:X_0\to X_1$ be the $\R$-points of $\bfalpha$.
Let $s_1:=\alpha(s_0)$, let $\bfs_1:=\bfalpha_*(\bfs_0)$ and let $S_1:=\bfs_1(\R)$.
Then $s_1=\alpha(s_0)\in\alpha_*(S_0)=S_1$.
Also, $S_1=\alpha_*(S_0)\subseteq\alpha_*(X_0)=X_1$.
Also, $\bfs_1=\bfalpha_*(\bfs_0)=\bfalpha_*(\Sm(\bfx_0))=\Sm(\bfalpha_*(\bfx_0))=\Sm(\bfx_1)$.
It suffices to~show that $S_1$ is Zariski dense in $\bfx_1$.
Let $\bfv$ denote the Zariski closure of $S_1$ in $\bfx_1$.
Assume, for a contradiction, that $\bfv\subsetneq\bfx_1$.

Let $V:=\bfv(\R)$.
Then $s_1\in S_1\subseteq V\subseteq X_1$.
Since $s_1\in V$, we get $V\ne\emptyset$, so $\bfv\ne\emptyset$.
Since $\bfv$ is Zariski closed in $\bfx_1$ and since $\bfx_1$~is Zariski closed in $\bfa^n$,
we conclude that $\bfv$ is Zariski closed in $\bfa^n$.
Let $d:=\dim\bfx_1$.
Then, since $\emptyset\ne\bfv\subsetneq\bfx_1$, since $\bfx_1$ is irreducible,
and since $\bfv$ and $\bfx_1$ are both Zariski closed in~$\bfa^n$,
it follows both that $d\ge1$ and that $\bfv$~is an $(n,d)$-variety.
Then, by \lref{lem-3-Rpts-variety-assembly},
we conclude that $V$~is an $(n,d)$-assembly.
Since $\bfs_1=\Sm(\bfx_1)$,
we see both that $\bfs_1$ is smooth and that $\bfs_1$ is Zariski open in $\bfx_1$.
Then $S_1$~is open in $X_1$.
So, as $S_1\subseteq V\subseteq X_1$,
$S_1$ is open in $V$.
Since $s_1\in S_1$, we get $S_1\ne\emptyset$, so $\bfs_1\ne\emptyset$.
Since $\bfs_1$ is nonempty and Zariski open in~$\bfx_1$ and since $\bfx_1$~is irreducible,
we see both that $\bfs_1$ is irreducible
and that $\dim\bfs_1=\dim\bfx_1$.
Also, since $\bfs_1$~is Zariski open in~$\bfx_1$
and since $\bfx_1$~is Zariski closed in $\bfa^n$,
we conclude that $\bfs_1$~is Zariski locally closed in~$\bfa^n$.
We have $\dim\bfs_1=\dim\bfx_1=d$.
Then, by~\lref{lem-1-Rpts-irred-submfld} (with $\bfs$~replaced by~$\bfs_1$),
we see that $S_1$ is a locally closed $d$-submanifold of~$\R^n$.
So, since $s_1\in S_1$, $S_1$~is $d$-dimensional in $\R^n$ near~$s_1$.
Since $S_1$~is open in~$V$ and $V$~has the relative topology inherited from~$\R^n$,
choose an~open subset $U$ of $\R^n$ such that $S_1=U\cap V$.
Since $U$ is open in~$\R^n$ and $s_1\in S_1=U\cap V\subseteq U$,
we get: $U$ is an open neighborhood of $s_1$ in~$\R^n$.
Choose an~$(\R^n,S_1,d)$-chart $\phi$ near~$s_1$ such that $\dom[\phi]\subseteq U$.
Define $D:=\dom[\phi]$.
Then $D=\dom[\phi]\subseteq U$, so $D\cap U=D$.
Then $D\cap S_1=D\cap U\cap V=D\cap V$.
Then we have
$$\phi_*(S_1)\quad=\quad\phi_*(D\cap S_1)\quad=\quad\phi_*(D\cap V)\quad=\quad\phi_*(V).$$
We conclude that $\phi$ is an~$(\R^n,V,d)$-chart near~$s_1$,
so $V$ is $d$-dimensional in~$\R^n$ near~$s_1$.
However, $V$ is an $(n,d)$-assembly, so, by~\lref{lem-5-ndassemb-not-ddim},
$V$~is {\it not} $d$-dimensional in $\R^n$ near $s_1$,
contradiction.
\end{proof}

\begin{cor}\wrlab{cor-7-one-smooth-Rpt-implies-Rpts-Zdense}
Let $\bfx$ be an irreducible $\R$-variety, $\bfs:=\Sm(\bfx)$.
Assume that $\bfs(\R)\ne\emptyset$.
Then $\widecheck{\bfx}=\bfx$, {\it i.e.}, $\bfx\in\widecheck{\scrv}$.
\end{cor}

\begin{proof}
Let $S:=\bfs(\R)$ and let $X:=\bfx(\R)$.
Then, by \lref{lem-6-sm-Rpt-imples-Zdense-Rpts},
$S$ is Zariski dense in~$\bfx$.
So, since $S\subseteq X$,
we see that $X$ is Zariski dense in~$\bfx$.
So, as $\widecheck{\bfx}$ is the Zariski closure of $X$ in $\bfx$,
we get $\widecheck{\bfx}=\bfx$.
\end{proof}

\begin{ex}\wrlab{ex-non-smooth-does-not-inherit}
$\exists$an irreducible $\R$-variety $\bfw$
s.t.~$\emptyset\ne\widecheck{\bfw}\subsetneq\bfw$.
\end{ex}

\noindent
{\it Details:}
Let $\bfw$ be the closed irreducible $\R$-subvariety of $\bfa^3$
defined by~$\bfw:=\{(x,y,z)\in\bfa^3\,|\,x^2+y^2+z^2=0\}$.
Then $\bfw(\R)=\{(0,0,0)\}$, so 
$\widecheck{\bfw}=\{(x,y,z)\in\bfa^3\,|\,x=y=z=0\}$.
Then $\emptyset\ne\widecheck{\bfw}\subsetneq\bfw$.
\qed

\vskip.1in
The next two examples show that smoothness is not something
that is easy to determine when only looking at $\R$-points.

\begin{ex}\wrlab{ex-realan-nonsmooth-Rpt}
In an irreducible closed $\R$-subvariety $\bfy$ of $\bfa^2$,
even if $\bfy(\R)$ is a $C^\omega$-submanifold of $\R^2$,
and even if $\bfy(\R)$ is Zariski dense in $\bfy$,
it does not follow that $\bfy(\R)\subseteq\Sm(\bfy)$.
\end{ex}

\noindent
{\it Details:}
Let $\bfy$ be the irreducible Zariski closed $\R$-subvariety of $\bfa^2$
defined by~$\bfy:=\{(x,y)\in\bfa^2\,|\,y^3=x^3(1+x^2)\}$.
Then $\bfy(\R)$ is the graph of the $C^\omega$ function $x\mapsto x(1+x^2)^{1/3}:\R\to\R$,
and it follows that $\bfy(\R)$ is a $C^\omega$-submanifold of $\R^2$.
Also, $\bfy(\R)$~is infinite, so, since $\bfy$~is $1$-dimensional and irreducible,
we see that $\bfy(\R)$ is Zariski dense in $\bfy$.
Finally, $(0,0)\in\bfy(\R)$ and $(0,0)\notin\Sm(\bfy)$,
so $\bfy(\R)\not\subseteq\Sm(\bfy)$.
\qed

\begin{ex}\wrlab{ex-smooth-Rpts-nonsmooth-var}
In an irreducible closed $\R$-subvariety $\bfz$ of $\bfa^2$,
even if $\bfz(\R)\subseteq\Sm(\bfz)$,
and even if $\bfz(\R)$ is Zariski dense in $\bfz$,
it does not follow that $\bfz$ is smooth.
\end{ex}

\noindent
{\it Details:}
Let $\bfz$ be the irreducible Zariski closed $\R$-subvariety of $\bfa^2$
defined by~$\bfz:=\{(x,y)\in\bfa^2\,|\,y^3=(1+x^2)^2\}$.
Then $\bfz(\R)$~is infinite and $\bfz$ is $1$-dimensional and irreducible,
so $\bfz(\R)$ is Zariski dense in~$\bfz$.
Let $\bfs:=\Sm(\bfz)$.
Then $\bfz\backslash\bfs=\{(x,y)\in\bfa^2\,|\,x^2=-1,y=0\}$.
Then $(\bfz\backslash\bfs)(\R)=\emptyset$,
and we conclude that $\bfz(\R)\subseteq\bfs(\R)\subseteq\bfs=\Sm(\bfz)$.
Also, since $\bfz\backslash\bfs\ne\emptyset$,
we see that $\bfz$ is not smooth.
\qed

\section{Results about real algebraic actions\wrlab{sect-real-alg-results}}

For any $\R$-variety $\bfv$,
recall: $\bfv(\R)$ is the topological space of $\R$-points of $V$,
together with the Hausdorff topology induced by~the standard Hausdorff topology on $\R$.
For any smooth $\R$-variety $\bfs$, 
$\bfs_M(\R)$~will denote the manifold of $\R$-points of $\bfs$.
For any algebraic $\R$-group~$\bfg$,
$\bfg_L(\R)$ will denote the Lie group of $\R$-points of $\bfg$.

Let $G$ be a Lie group acting on a topological space $X$.
We will say that the action is {\bf real algebraic} if there exist
\begin{itemize}
\item an algebraic $\R$-group $\bfg$,
\item an $\R$-variety $\bfv$ \qquad\qquad and
\item an $\R$-action of $\bfg$ on $\bfv$
\end{itemize}
such that $G=\bfg_L(\R)$, such that $X=\bfv(\R)$
and such that the action of~$G$ on $X$ is equal to
the $\R$-points of the $\R$-action of $\bfg$ on $\bfv$.

Let $G$ be a Lie group acting on a manifold $M$.
We will say that the action is {\bf smoothly real algebraic} if
there exist
\begin{itemize}
\item an algebraic $\R$-group $\bfg$,
\item a smooth $\R$-variety $\bfs$ \qquad\qquad and
\item an $\R$-action of $\bfg$ on $\bfs$
\end{itemize}
such that $G=\bfg_L(\R)$, such that $M=\bfs_M(\R)$
and such that the action of $G$ on $M$ is equal to
the $\R$-points of the action of $\bfg$ on $\bfs$.
If the $G$-action on $M$ is smoothly real algebraic,
then it is $C^\omega$.

By Proposition 1.6, p.~8, of \cite{coste:realalgsets},
for any $\R$-variety $\bfv$,
the Hausdorff topological space $\bfv(\R)$ has only finitely many connected components.
Also, for any algebraic $\R$-group $\bfg$,
for~any open subgroup $G_\circ$ of $\bfg_L(\R)$,
the index of $G_\circ$ in $\bfg_L(\R)$ is finite.
Also, any stabilizer in a real algebraic action
has only finitely many connected components.

We make a small improvement to Lemma~6.1 in~\cite{adams:algacts}:

\begin{lem}\wrlab{lem-alg-cpt-stabs-to-proper}
Let a Lie group $G$ have a real algebraic action on a topological space $X$.
Let $G_\circ$~be an open subgroup of $G$.
Let $X_\circ$ be an open $G_\circ$-invariant subset of $X$.
Assume: for all $x\in X_\circ$, $\Stab_{G_\circ}(x)$ is compact.
Then there exists a dense open $G_\circ$-invariant subset $W$ of $X_\circ$
such that~the $G_\circ$-action on~$W$ is proper.
\end{lem}

\begin{proof}
Let~$\scru$ denote the set of all open $G$-invariant subsets $Y$ of~$GX_\circ$
such that the $G$-action on~$Y$ is proper.
By convention, we have $\emptyset\in\scru$.
Let $\scru_*:=\{\scrw\subseteq\scru\,|\,\forall Y,Z\in\scrw,[(Y\cap Z=\emptyset)\hbox{ or }(Y=Z)]\}$
denote the set of all pairwise-disjoint subsets of $\scru$.
Let $\scru_*$ be partially ordered by~inclusion.
By Zorn's Lemma, choose a maximal element $\scrw_1\in\scru_*$.
For all~$\scrw\in\scru_*$, we have $\bigcup\scrw\in\scru$.
Let $W_1:=\bigcup\scrw_1$.
Then $W_1\in\scru$,
so $W_1$ is an open $G$-invariant subset of~$GX_\circ$,
and the $G$-action on~$W_1$ is proper.
Let $W:=W_1\cap X_\circ$.
Then $W$ is an open $G_\circ$-invariant subset of~$X_\circ$.
Since $G_\circ$ is an open subgroup of~$G$, $G_\circ$ is closed in $G$.
So, since the $G$-action on~$W_1$ is proper,
the $G_\circ$-action on $W_1$ is proper.
So, since $W\subseteq W_1$, the $G_\circ$-action on $W$ is proper.
It remains to show: $W$~is dense in~$X_\circ$.
Since $X_\circ$ is open in $X$, $X_\circ$ is open in $GX_\circ$.
So, as $W=W_1\cap X_\circ$,
it suffices to show that $W_1$ is dense in $GX_\circ$.
Let $\overline{W_1}$~denote the closure of~$W_1$ in~$GX_\circ$.
We wish to show: $\overline{W_1}=GX_\circ$.
Let $X_1:=(GX_\circ)\backslash\overline{W_1}$.
Assume that $X_1\ne\emptyset$.
We aim for a contradiction.

Since the index of $G_\circ$ in $G$ is finite, we see, for all $x\in X$, that
the index of~$\Stab_{G_\circ}(x)$ in $\Stab_G(x)$ is finite.
By assumption, for all~$x\in X_\circ$, $\Stab_{G_\circ}(x)$ is compact.
Therefore, for all $x\in X_\circ$, $\Stab_G(x)$ is compact.
Therefore, for all $x\in GX_\circ$, $\Stab_G(x)$ is compact.
So, as $X_1\subseteq GX_\circ$, we see, for all~$x\in X_1$, that $\Stab_G(x)$ is compact.
Since $X_\circ$ is open in~$X$,
it follows that $GX_\circ$ is open in $X$.
Since $\overline{W_1}$ is closed in $GX_\circ$,
we see that $X_1$ is open in $GX_\circ$.
So, since $GX_\circ$ is open in~$X$,
$X_1$ is open in~$X$.
By Lemma~6.1 in~\cite{adams:algacts}
(with $V$ replaced by~$X$ and $V_0$~by~$X_1$),
choose a~nonempty open $G$-invariant subset $U$ of $X_1$
such that the $G$-action on~$U$ is proper.
Because $U$~is open in $X_1$
and $X_1$ is open in $GX_\circ$,
we see that $U$~is open in~$GX_\circ$.
Then we have $U\in\scru$.
Also, we have $U\cap(\bigcup\scrw_1)=U\cap W_1\subseteq X_1\cap\overline{W_1}=((GX_0)\backslash\overline{W_1})\cap\overline{W_1}=\emptyset$.
Since $\scrw_1\in\scru_*$,
we see that $\scrw_1\subseteq\scru$ and that $\scrw_1$~is pairwise-disjoint.
Then $\scrw_1\cup\{U\}\subseteq\scru$
and $\scrw_1\cup\{U\}$ is pairwise-disjoint.
Then $\scrw_1\cup\{U\}\in\scru_*$,
contradicting maximality of~$\scrw_1$.
\end{proof}

\begin{thm}\wrlab{thm-mainalg}
Let a Lie group $G$ have a smoothly real algebraic action on a manifold $M$.
Let $G_\circ$~be an open subgroup of $G$.
Let $M_\circ$ be a~nonempty open $G_\circ$-invariant subset of $M$.
Assume that the $G_\circ$-action on $M_\circ$ is component effective.
Define
$$n\,\,:=\,\,\dim G\quad\qquad\hbox{and}\qquad\quad\ell\,\,:=\,\,\begin{cases}
n,&\hbox{if }n\le1;\\
n-1,&\hbox{if }n\ge2.\end{cases}$$
Then there exists a dense open $G_\circ$-invariant subset $W$ of~$\scrf^\ell M_\circ$
such~that the $G_\circ$-action on $W$ is free and proper.
\end{thm}

\begin{proof}
{\it Claim:}
There is a dense open $G_\circ$-invariant subset $X_\circ$ of $\scrf^\ell M_\circ$
such that the $G_\circ$-action on $X_\circ$ is free.
{\it Proof of Claim:}
The $G$-action on $M$ is smoothly real algebraic,
and therefore $C^\omega$.
Then the $G_\circ$-action on~$M_\circ$ is also $C^\omega$.
Then, by \lref{lem-real-analytic-component-eff-iff-fixpt-rare}
(with $G$ replaced by $G_\circ$ and $M$ by $M_\circ$),
the $G_\circ$-action on $M_\circ$ is~fixpoint rare.

For all $g\in G_\circ\backslash\{1_G\}$, let $S_g:=\{q\in M_\circ\,|\,gq\ne q\}$;
then $S_g$ is dense open in $M_\circ$.
If $n=0$, then $G$ is finite and $\ell=0$, and,
identifying $\scrf^0M_\circ$ with $M_\circ$,
we may set $\displaystyle{X_\circ:=\bigcap\,\{S_g\,|\,g\in G_\circ\backslash\{1_G\}\}}$.
We assume $n\ge1$.

Let $\pi:\scrf^\ell M_\circ\to M_\circ$ be the structure map of $\scrf^\ell M_\circ$.
By definition of~$\ell$, $\ell\in\{n-1,n\}$, so $\ell\ge n-1$.
Let $G^\circ$ denote the identity component of~$G_\circ$.
The $G_\circ$-action on $M_\circ$ is~fixpoint rare and $G^\circ\subseteq G_\circ$,
and so the $G^\circ$-action on~$M_\circ$ is fixpoint rare.
Then, by Theorem 3.3 of \cite{adams:genfreeacts}
(with $G$~replaced by~$G_\circ$, $M$~by~$M_\circ$),
choose a dense open $G_\circ$-invariant subset~$M_1$ of $M_\circ$
such that the $G_\circ$-action on~$\pi^*(M_1)$ is locally free.
Let $X_\circ:=\pi^*(M_1)$.
Then the $G_\circ$-action on~$X_\circ$ is locally free.
Since $\pi:\scrf^\ell M_\circ\to M_\circ$ is open, continuous and $G_\circ$-equivariant
and since $M_1$~is a dense open $G_\circ$-invariant subset of~$M_\circ$,
we conclude that $X_\circ$ is a dense open $G_\circ$-invariant subset of~$\scrf^\ell M_\circ$.
It only remains to show that the $G_\circ$-action on $X_\circ$ is free.

As $n\ge1$, we get $\ell\ge1$, so $\ell\in\N$.
Then, by~\cref{cor-cpt-stabs-in-subset-prolong} (with $G$ replaced by $G_\circ$, $M$ by $M_\circ$ and $V$ by~$X_\circ$),
it suffices to show, for all $x\in X_\circ$, that $\Stab_{G_\circ}(x)$ is compact.
Let $x\in X_\circ$ be given, and let~$S_\circ:=\Stab_{G_\circ}(x)$.
We wish to show that $S_\circ$ is compact.

Since the $G_\circ$-action on~$X_\circ$ is locally free,
we see that $S_\circ$ is discrete.
Let $S:=\Stab_G(x)$.
Since the index of $G_\circ$ in $G$ is finite, we see that the index of~$S_\circ$ in $S$ is finite.
So, since $S_\circ$ is discrete, $S$~is discrete.
The $G$-action on $M$ is smoothly real algebraic,
so the~$G$-action on $\scrf^\ell M$ is also smoothly real algebraic,
so $S$ has only finitely many connected components.
So, since $S$ is discrete, $S$ is finite.
So, since $S_\circ\subseteq S$, we see that $S_\circ$ is finite, and so $S_\circ$ is compact.
{\it End of proof of claim.}

We identify $\scrf^\ell M_\circ$ with the preimage of $M_\circ$
under the structure map $\scrf^\ell M\to M$.
Then $\scrf^\ell M_\circ$ is an open subset of~$\scrf^\ell M$.
Choose $X_\circ$ as in~the claim.
Since $X_\circ$ is open in $\scrf^\ell M_\circ$
and $\scrf^\ell M_\circ$~is open in~$\scrf^\ell M$,
it follows that $X_\circ$~is open in $\scrf^\ell M$.
The $G_\circ$-action on $X_\circ$ is free, so, for all~$x\in X_\circ$,
we have $\Stab_{G_\circ}(x)=\{1_G\}$,
so $\Stab_{G_\circ}(x)$ is compact.
Then, by \lref{lem-alg-cpt-stabs-to-proper},
choose a dense open $G_\circ$-invariant subset $W$ of~$X_\circ$
such that the $G_\circ$-action on~$W$ is proper.
Because $W\subseteq X_\circ$ and because the $G_\circ$-action on~$X_\circ$ is free,
we see that the $G_\circ$-action on $W$ is also free.
It remains to show that $W$ is dense open in $\scrf^\ell M_\circ$.

Since $W$ is dense open in $X_\circ$,
and since, by the claim, $X_\circ$ is dense open in $\scrf^\ell M_\circ$,
it follows that $W$ is dense open in $\scrf^\ell M_\circ$, as desired.
\end{proof}

\begin{ex}\wrlab{ex-ell-is-sharp}
In \tref{thm-mainalg}, it is important that we define $\ell$
in~such a way that if $n=1$, then $\ell\ne0$.
\end{ex}

\noindent
{\it Details:}
Let~$G$~be the group of isometries (reflections and translations) of $\R$.
Let $G_\circ:=G$, $M:=\R$ and $M_\circ:=\R$.
The $G$-action on~$M$ is smoothly real algebraic, and
the $G_\circ$-action on~$M_\circ$ is component effective,
but every point of $M_\circ$ is fixed by a reflection in $G_\circ$.
Thus $G_\circ$~does not act freely on any
nonempty $G_\circ$-invariant subset of $\scrf^0M_\circ$.
\qed

\section{Main results\wrlab{sect-results}}

Let a Lie group $G$ have a $C^\infty$ action on a manifold $M$.
For all $k\in\N_0$, let $\pi_k:\scrf^kM\to M$ be the structure map of $\scrf^kM$.

\begin{thm}\wrlab{thm-freeAb-cent-stabs}
Assume that $G$ is connected and that the $G$-action on $M$ is fixpoint rare.
Let $n:=\dim G$.
Then there exists a dense open $G$-invariant subset $M_1$ of $M$
such that, for all $x\in\pi_n^*(M_1)$,
$\Stab_G(x)$~is~a~discrete, finitely-generated, free-Abelian subgroup of $Z(G)$.
\end{thm}

\begin{proof}
If $n=0$, then, because $G$ is connected, it follows that $G$ is trivial,
and we may set $M_1:=M$.
We assume $n\ge1$.
Then $n\in\N$.

For all $k\in\N_0$, for all $x\in\scrf^kM$, let $G_x:=\Stab_G(x)$.

By Theorem 3.3 of \cite{adams:genfreeacts}
(with $\ell=n-1$ and $G^\circ=G$),
choose a~dense open $G$-invariant subset $M_1$ of $M$
s.t.~the $G$-action on $\pi_{n-1}^*(M_1)$ is locally free.
Let $x\in\pi_n^*(M_1)$ be given.
We wish to~prove that $G_x$ is a discrete, finitely-generated, free-Abelian subgroup of $Z(G)$.

Let $X:=\scrf^nM$, $W:=\scrf^{n-1}M$, $X_1:=\pi_n^*(M_1)$  and $W_1:=\pi_{n-1}^*(M_1)$.
Let $\sigma:X\to W$ be the canonical map.
Let~$w:=\sigma(x)$.
Then, as~$\sigma$~is $G$-equivariant, we get $G_x\subseteq G_w$.
Since $\pi_{n-1}\circ\sigma=\pi_n$, we get $\pi_{n-1}(w)=\pi_n(x)$.
Then $\pi_{n-1}(w)=\pi_n(x)\in\pi_n(\pi_n^*(M_1))\subseteq M_1$,
and so $w\in\pi_{n-1}^*(M_1)$.
Therefore, since the $G$-action on~$\pi_{n-1}^*(M_1)$ is locally free,
we conclude that $G_w$~is a discrete subgroup of~$G$.

{\it Claim:} $G_x\subseteq Z(G)$.
{\it Proof of claim:}
Let $a\in G_x$ be given.
We wish to show that $a\in Z(G)$.
Let $\Lg:=T_{1_G}G$.
Since $G$ is connected,
$Z(G)$~is the kernel of the Adjoint representation $\Ad:G\to\GL(\Lg)$.
It therefore suffices to show that $\Ad\,a:\Lg\to\Lg$ is equal to the identity map $\Lg\to\Lg$.
Let $B\in\Lg$ be given.
We wish to show that $(\Ad\,a)B=B$.

Let $\iota:\Lg\to T_wW$ be the differential, at $1_G$, of $g\mapsto gw:G\to W$.
Since $G_w$ is discrete, by Corollary 5.2 of \cite{adams:locfreeacts}, $\iota:\Lg\to T_wW$ is injective.
Let $v:=\iota(B)$.
Then $v\in T_wW$.
By Lemma 3.1 of \cite{adams:genfreeacts}
(with $k$ replaced by~$n$),
$\Stab_G(x)\subseteq\Stab_G(v)$.
Then $a\in G_x=\Stab_G(x)\subseteq\Stab_G(v)$,
and so $av=v$.
Let $G$ act on $\Lg$ via the Adjoint representation.
Then the map $\iota:\Lg\to T_wW$ is $G_w$-equivariant.
So, since $a\in G_x\subseteq G_w$ and $\iota(B)=v$,
we get $\iota((\Ad\,a)B)=av$.
Then $\iota((\Ad\,a)B)=av=v=\iota(B)$,
so, by~injectivity of $\iota:\Lg\to T_wW$,
$(\Ad\,a)B=B$.
{\it End of proof of claim.}

Since $G_x\subseteq G_w$ and since $G_w$ is discrete,
we conclude that $G_x$ is also discrete.
So, by the claim, $G_x$~is a discrete subgroup of~$Z(G)$.
So, by~Lemma~3.1 in \cite{adolv:generic},
$G_x$ is finitely-generated.
It~remains only to~show that $G_x$ is free-Abelian.
As $G_x$ is finitely-generated and Abelian,
it suffices to show that $G_x$ is torsion-free.

The $G$-action on $M$ is fixpoint rare, and, therefore, component effective.
Also, $x\in\pi_n^*(M_1)\subseteq\scrf^nM$.
So, since $n\in\N$ and since $G_x=\Stab_G(x)$,
by~\lref{lem-no-cpt-subgps-in-stabs} (with $\ell$ replaced by $n$),
we see that $G_x$~has no nontrivial compact subgroups.
Then $G_x$ is torsion-free.
\end{proof}

\begin{cor}\wrlab{cor-cpt-center}
Assume that $G$ is connected, that $Z(G)$ is compact and that the $G$-action on $M$ is fixpoint rare.
Let $n:=\dim G$.
Then there exists
a dense open $G$-invariant subset $M_1$ of $M$
such that the $G$-action on~$\pi_n^*(M_1)$ is free.
\end{cor}

\begin{proof}
Choose $M_1$ as in \tref{thm-freeAb-cent-stabs}.
Let $x\in\pi_n^*(M_1)$ be given.
We wish to show that $\Stab_G(x)=\{1_G\}$.

By \tref{thm-freeAb-cent-stabs},
we know that $\Stab_G(x)$ is a discrete, free-Abelian subgroup of~$Z(G)$.
Since $Z(G)$ is compact,
any discrete subgroup of~$Z(G)$ is finite, and so $\Stab_G(x)$~is finite.
A finite, free-Abelian group is trivial, and so $\Stab_G(x)=\{1_G\}$, as desired.
\end{proof}


\bibliography{list}

\begin{thebibliography}{Zeg95b}

\bibitem[Ad1]{adams:cinftyctx}
S.~Adams.
\newblock Freeness in higher order frame bundles.
\newblock Unpublished, 2014.
\newblock arXiv:1509.01609

\bibitem[Ad2]{adams:comegactx}
S.~Adams.
\newblock Real analytic counterexample to the freeness conjecture.
\newblock Unpublished, 2014.
\newblock arXiv:1509.01607

\bibitem[Ad3]{adams:algacts}
S.~Adams.
\newblock Moving frames and prolongations of real algebraic actions.
\newblock Unpublished, 2014.
\newblock arXiv:1305.5742

\bibitem[Ad4]{adams:locfreeacts}
S.~Adams.
\newblock Local freeness in frame bundle prolongations of $C^\infty$ actions.
\newblock Unpublished, 2017.
\newblock arXiv:1608.06595

\bibitem[Ad5]{adams:genfreeacts}
S.~Adams.
\newblock Generic freeness in frame bundle prolongations of $C^\infty$ actions.
\newblock Unpublished, 2017.
\newblock arXiv:1605.06527

\bibitem[AdOlv]{adolv:generic}
S.~Adams and P.~J.~Olver.
\newblock Prolonged analytic connected group actions are generically free.
\newblock To appear, {\it Transformation Groups}.

\bibitem[Coste]{coste:realalgsets}
M.~Coste.
\newblock Real Algebraic Sets.\\
\newblock https://perso.univ-rennes1.fr/michel.coste/polyens/RASroot.pdf.

\bibitem[Olver]{olver:movfrmsing}
P.~J.~Olver.
\newblock Moving frames and singularities of prolonged group actions.
\newblock Selecta Math.~{\bf6} (2000), 41--77.

\end{thebibliography}

\end{document}